\documentclass[11pt]{amsart}

\usepackage{url}
\usepackage{amsmath}
\usepackage{graphicx}
\usepackage{amssymb}
\usepackage[T1]{fontenc}
\usepackage{tikz}
\usepackage{enumerate}

\newcommand*{\rom}[1]{\expandafter\@slowromancap\romannumeral #1@}
\newcommand{\RNum}[1]{\uppercase\expandafter{\romannumeral #1\relax}}
\newcommand{\N}{\mathbb{N}}  

\theoremstyle{plain} 
\newtheorem{thm}{Theorem}[section]
\newtheorem{lem}{Lemma}[section]
\newtheorem{pro}{Proposition}[section]

\newtheorem{qu}{Question}

\theoremstyle{definition} 
\newtheorem{defn}{Definition}[section]

\theoremstyle{remark} 

\newtheorem*{ack}{Acknowledgements}

\title{An inverse of Furstenberg's correspondence principle and applications to nice recurrence}
\author{Alexander Fish and Sean Skinner}

\address{School of Mathematics and Statistics, University of Sydney, Australia}
\curraddr{}
\email{alexander.fish@sydney.edu.au}
\email{sean.skinner@sydney.edu.au}
\thanks{}

\date{\today}                                           

\begin{document}
\begin{abstract}
We prove an inverse of Furstenberg's correspondence principle stating that for all measure preserving systems $(X,\mu,T)$ and $A\subset X$ measurable there exists a set $E \subset \N$ such that
\[ \mu\left( \bigcap_{i=1}^k T^{-n_i}A\right) = \lim_{N\to \infty} \frac{\left|\left( \bigcap_{i=1}^k (E-n_i) \right)\cap \{0,\dots,N-1\}\right|}{N}\]
for all $k,n_1,\dots,n_k \in \N$. As a corollary we show that a set $R\subset\N$ is a set of nice recurrence if and only if it is nicely intersective. Together, the inverse of Furstenberg's correspondence principle and it's corollary partially answer two questions of Moreira.
\end{abstract}

\maketitle
\section{\textbf{Introduction}}
Poincar\'{e}'s famous recurrence theorem says that for any measure preserving system $(X,\mu,T)$ and any $A\subset X$ measurable\footnote{We choose not to include the underlying $\sigma$-algebra in our notation and moving forward all subsets of a measurable space will be assumed to be measurable.} with $\mu(A)>0$ there exists a positive integer $r$ for which
\[\mu(A\cap T^{-r}A)>0.\]
A deeper result of Furstenberg \cite{Fur2} states that under the same hypotheses one can always find a positive integer $r$ for which
\begin{equation} \label{eq: Squares are set of recurrence}
\mu(A\cap T^{-r^2}A)>0.
\end{equation}
This phenomenon motivates the notion of a set of recurrence as introduced by Furstenberg.
\begin{defn}[Sets of recurrence]\label{Def: Recurrence}
A set $R\subset \N$ is called a set of recurrence if for any measure preserving system $(X,\mu,T)$ and any set $A\subset X$ with $\mu(A)>0$ there exists some non-zero $r\in R$ for which
\[ \mu(A \cap T^{-r} A) > 0.\]
\end{defn}
Thus Poincar\'{e}'s theorem states that the positive integers are a set of recurrence and the content of equation \eqref{eq: Squares are set of recurrence} states that the set of square numbers is a set of recurrence. Of course by now we know of many more examples of sets of recurrence\footnote{See for example \cite{Fur2} and \cite{KM}}, however in this paper we will be focused on the relationship between results in recurrence and their counterparts in density Ramsey theory. Independently to Furstenberg, Sark\"{o}zy \cite{Sa} proved that if $E\subset \N$ has positive upper density,
\begin{equation*}\label{eq: Upper density}
\overline{d}_{(\{0,\dots,N-1\})}(E) : = \limsup_{N\to \infty} \frac{|E\cap \{0,1,\dots,N-1\}|}{N}>0,
\end{equation*}
then there must exist some non-zero integer $r$ for which
\begin{equation} \label{eq: Squares are intersective}
E\cap (E-r^2) \neq \emptyset,
\end{equation}
or in other words, $E$ must contain two elements that differ by a perfect square. The above also holds when $(\{0,\dots,N-1\})$ is replaced by any other F{\o}lner sequence in $\N$.
\begin{defn}[F{\o}lner sequences]
A sequence $(F_N)$ of finite subsets of $\N$ is a F{\o}lner sequence if for any $t\in \N$ we have that
\begin{equation*}
\frac{|F_N \Delta  (F_N+t)|}{|F_N|} \to 0 \text{ as } N\to \infty.
\end{equation*}
For a F{\o}lner sequence $(F_N)$ and a set $E\subset \N$ we define the upper density of $E$ with respect to $(F_N)$ to be
\begin{equation*}
\overline{d}_{(F_N)}(E) := \limsup_{N\to \infty}\frac{|E\cap F_N|}{|F_N|}.
\end{equation*}
When the limit supremum is actually a limit we drop the overline and call $d_{(F_N)}(E)$ the density of $E$ with respect to $(F_N)$.
\end{defn}
\begin{defn}[Intersective Sets]
Let $(F_N)$ be a F{\o}lner sequence in $\N$ and let $R \subset \N$. $R$ is $(F_N)$-intersective if for every $E\subset \N$ with $\overline{d}_{(F_N)}(E)>0$ there exists some non-zero $r\in R$ with
\begin{equation*}
E\cap(E-r)\neq \emptyset.
\end{equation*}
\end{defn}
Even though one can easily verify that different F{\o}lner sequences can define genuinely different upper densities\footnote{Consider for example the sequence $(\{N^2+1,\dots,N^2+N\})$.}, it turns out that the corresponding notions of intersectivity are all equivalent, and in fact we have the following characterisation which is implicit in \cite{Fur2}, and explicit in \cite{Moreira}.
\begin{thm}[Characterisation of sets of recurrence]\label{Theorem: Characterising sets of recurrence}
Let $R\subset \N$. Then $R$ is a set of recurrence if and only if it is $(F_N)$-intersective for some F{\o}lner sequence $(F_N)$.
\end{thm}
Thus, the results referred to in equations \eqref{eq: Squares are set of recurrence} and \eqref{eq: Squares are intersective} are equivalent and jointly referred to as the Furstenberg-Sark\"{o}zy theorem.

The only if direction of Theorem \ref{Theorem: Characterising sets of recurrence} follows immediately from
Furstenberg's now ubiquitous correspondence principle.
\begin{thm}[Furstenberg's Correspondence Principle \cite{Fur2}] \label{Theorem: FCP}
Given any $E\subset \N$ and any F{\o}lner sequence $(F_N)$ there exists a measure preserving system $(X,\mu,T)$, a set $A\subset X$ with $\mu(A) = \overline{d}_{(F_N)}(E)$, and a subsequence $(N_i)$ such that
\[ \mu(A\cap T^{-n_1}A\cap \dots \cap T^{-n_k}A) = d_{(F_{N_i})}(E\cap (E-n_1)\cap \dots \cap (E-n_k)) \]
for all $k,n_1,\dots,n_k \in \N$.
\end{thm}
In 1986 Bergelson introduced a different notion of recurrence and intersectivity as part of his proof of a density version of Schur's theorem\cite{B}.
\begin{defn}[Nice Recurrence] \label{Def: Nice Recurrence}
A set $R\subset \N$ is called a set of nice recurrence (c.f. \cite{B}) if for any measure preserving system $(X,\mu,T)$, any $A\subset X$, and any $\varepsilon>0$ there exists some non-zero $r\in R$ such that
\[\mu(A\cap T^{-r}A)>\mu(A)^2 -\varepsilon.\]
\end{defn}
Note that for any strong mixing system $\mu(A\cap T^{-n}A) \to \mu(A)^2$ as $n\to \infty$ and so the lower bound in  Definition \ref{Def: Nice Recurrence} is in this sense optimal. Nearly all known examples of sets of recurrence (including the squares) are also sets of nice recurrence, however in \cite{ForThesis} Forrest constructed a set of recurrence which is not a set of nice recurrence.
\begin{defn}[Nicely Intersective Sets] \label{Def: Nicely Intersective Sets}
Let $(F_N)$ be a F{\o}lner sequence in $\N$ and let $R \subset \N$. We call $R$ a $(F_N)$-nicely intersective set (also referred to as a set of \textit{nice combinatorial recurrence} in \cite{B} for the case $(F_N)=(\{1,\dots,N\})$) if for any set $E\subset \N$ and any $\varepsilon>0$ there exists some non-zero $r\in R$ such that
\[\overline{d}_{(F_N)}(E\cap (E-r)) > \overline{d}_{(F_N)}(E)^2 -\varepsilon.\]
\end{defn}
It is natural to ask whether a characterisation similar to Theorem \ref{Theorem: Characterising sets of recurrence} holds for sets of nice recurrence. Indeed Theorem \ref{Theorem: FCP} immediately implies that for any F{\o}lner sequence $(F_N)$, any set of nice recurrence $R$ is $(F_N)$-nicely intersective. A question of Moreira in \cite{Moreira} asks if the converse is also true.
\begin{qu}\label{Question: Characterising nice recurrence}
Is it true that a set $R\subset \N$ is a set of nice recurrence if and only if it is $(F_N)$-nicely intersective for some F{\o}lner sequence $(F_N)$?
\end{qu}
We provide a partial answer to this question.
\begin{thm}\label{Theorem: Characterisation of sets of nice recurrence}
A set $R\subset \N$ is a set of nice recurrence if and only if it is $(\{0,\dots,N-1\})$-nicely intersective.
\end{thm}
In particular, by Theorem \ref{Theorem: FCP} we have that any $(\{0,\dots,N-1\})$-nicely intersective set is also $(F_N)$-nicely intersective set for all F{\o}lner sequences $(F_N)$. 
Moreira also observed that an affirmative answer to Question \ref{Question: Characterising nice recurrence} would follow from an affirmative answer to the following question.
\begin{qu}\label{Question: Inverse CP}
Let $(X,\mu,T)$ be a measure preserving system and let $A\subset X$. Is it true that for any F{\o}lner sequence $(F_N)$ there is some set $E\subset \N$ such that for every $n\in \N$ we have that
\[\mu(A\cap T^{-n} A) = \overline{d}_{(F_N)}(E\cap(E-n)).\]
\end{qu}
Indeed our proof of Theorem \ref{Theorem: Characterisation of sets of nice recurrence} follows from the following partial answer to Question \ref{Question: Inverse CP}.
\begin{thm} \label{Theorem: Inverse CP}
Given any measure preserving system $(X,\mu,T)$ and any set $A\subset X$ there exists some $E\subset \N$ such that
\[ \mu\left( \bigcap_{i=1}^k T^{-n_i}A\right) = \lim_{N\to \infty} \frac{\left|\left( \bigcap_{i=1}^k (E-n_i) \right)\cap \{0,\dots,N-1\}\right|}{N}\]
for any $k,n_1,\dots,n_k \in \N$.
\end{thm}
In a slightly different direction, it is also known that sets of recurrence admit an equivalent uniform characterisation.
\begin{thm}[Uniformity of recurrence \cite{ForThesis}]\label{Theorem: Uniformity of recurrence}
Let $R\subset \N$ be a set of recurrence. Then for any $\delta>0$ there exists some positive integer $N_0=N_0(R,\delta)$ and $\varepsilon = \varepsilon(R,\delta)$ such that for any measure preserving system $(X,\mu,T)$ and any $A\subset X$ with $\mu(A)\geq \delta$ there exists some $r\in R \cap \{1,\dots,N_0\}$ with \[\mu(A\cap T^{-r}A)\geq \varepsilon.\]
\end{thm}
An analogous result also holds for sets of nice recurrence.
\begin{thm}[Uniformity of nice recurrence]\label{Theorem: Uniformity of nice recurrence}
Let $R\subset \N$ be a set of nice recurrence. Then for any $\delta,\varepsilon>0$ there exists some positive integer $N_0=N_0(R,\delta,\varepsilon)$ such that for any measure preserving system $(X,\mu,T)$ and any $A\subset X$ with $\mu(A)\geq \delta$ there exists some $r \in R \cap \{1,\dots,N_0\}$ with \[\mu(A\cap T^{-r} A)>\delta^2 - \varepsilon.\]
\end{thm}
Although there are now at least 4 different proofs of Theorem \ref{Theorem: Uniformity of recurrence} in the literature (c.f. \cite{ForThesis,FraKra,BHMP,FLW,BH,Ack}), the scrupulous reader can verify that only the arguments of Forrest's original proof also suffice to prove Theorem \ref{Theorem: Uniformity of nice recurrence}. To our best knowledge this is the first time that Theorem \ref{Theorem: Uniformity of nice recurrence} has appeared in print, and in any case, we provide a new and simpler proof of both Theorem \ref{Theorem: Uniformity of recurrence} and Theorem \ref{Theorem: Uniformity of nice recurrence}.

\begin{ack}
We are grateful to John Griesmer for useful discussions on the topic of this paper.
A. Fish was supported by the ARC via grants DP210100162
and DP240100472.
\end{ack}
\section{\textbf{Proof of theorem \ref{Theorem: Characterisation of sets of nice recurrence}}}
For convenience let us denote $(G_N):=(\{0,\dots,N-1\})$.
To prove the 'if' direction of Theorem \ref{Theorem: Characterisation of sets of nice recurrence} suppose that $R\subset \N$ is a $(G_N)$-nicely intersective set and let $(X,\mu,T)$ be a measure preserving system with $A\subset X$. By Theorem \ref{Theorem: Inverse CP} there exists some $E\subset \N$ with
\[d_{(G_N)}(E\cap (E-t)) = \mu(A\cap T^{-t}) \quad \text{for all }t \in \N.\]
In particular $d_{(G_N)}(E)=\mu(A)$ and since $R$ is a $(G_N)$-nicely intersective set then for every $\varepsilon>0$ there exists some non-zero $r\in R$ with
\[\mu(A)^2-\varepsilon=d_{(G_N)}(E)^2-\varepsilon < d_{(G_N)}(E\cap(E-r)) = \mu(A\cap T^{-r}A)\]
as required. 

On the other hand is $R\subset \N$ is a set of nice recurrence and $E\subset \N$, then for any F{\o}lner sequence $(F_N)$ Theorem \ref{Theorem: FCP} ensures the existence of a measure preserving system $(X,\mu,T)$ and a set $A \subset X$ with $\mu(A) = \overline{d}_{(F_N)}(E)$ and 
\[\mu(A\cap T^{-t}A)\leq \overline{d}_{(F_N)}(E \cap(E-t)) \quad \text{for any }t\in \N.\]
Since $R$ is a set of nice recurrence then for every $\varepsilon>0$ there exists some non-zero $r\in R$ with
\[ \overline{d}_{(F_N)}(E)^2 - \varepsilon=\mu(A)^2 -\varepsilon < \mu(A\cap T^{-r}A) \leq \overline{d}_{(F_N)}(E \cap(E-r))\]
as required.
\section{\textbf{Proof of the inverse correspondence principle}}
In what follows the space $\{0,1\}^\N$ of binary valued sequences indexed by $\N$ will always be equipped with the product topology (where $\{0,1\}$ is discrete), the Borel $\sigma$-algebra and the right shift $\sigma: \{0,1\}^\N \to \{0,1\}^\N$ defined by
\[(\sigma \omega )_n = \omega_{n+1} \text{ for any }\omega \in \{0,1\}^\N \text{ and any } n\in \N.\]
For any $r, n_1,\dots,n_r \in \N$ and $i_1,\dots,i_r \in \{0,1\}$ we denote by
\[[\omega_{n_1} = i_1,\dots,\omega_{n_r}=i_r]\]
the set of all sequences $\omega \in \{0,1\}^\N$ with $\omega_{n_j} = i_j$ for each $j=1,\dots,r$. Such a set is called a cylinder. 
Recall that the collection of cylinders forms a generating algebra for the Borel $\sigma$-algebra on $\{0,1\}^\N$. Denote by $\mathcal{M}$ the compact space of all Borel probability measures on $\{0,1\}^\N$ equipped with the topology of weak-* convergence, which in this setting means that a sequence of measures $\nu_n \in \mathcal{M}$ converges to some $\nu \in \mathcal{M}$ if and only if
\begin{equation*}
\nu_n(C) \to \nu(C) \quad \text{for all cylinders } C.
\end{equation*}
Denote by $\mathcal{M}_\sigma$ the set of all $\sigma$-invariant measures in $\mathcal{M}$.
We begin with the well known observation that $(\{0,1\}^\N,\sigma)$ is universal for all measure preserving systems $(X,\mu,T)$ with a distinguished set $A\subset X$ in the following sense.
\begin{lem}[\cite{Tao}{[Example 2.2.6]}]\label{Lemma: Universality of shift space}
Let $(X,\mu,T)$ be a measure preserving system. Given any $A\subset X$ there exists some $\nu \in \mathcal{M}_\sigma$ such that for any $r, n_1,\dots, n_r \in \N$ and $i_1,\dots,i_r \in \{0,1\}$,
\[\nu([\omega_{n_1} = i_1,\dots, \omega_{n_r} = i_r]) = \mu(T^{-n_1} A^{i_1} \cap \dots \cap T^{-n_r} A^{i_r}),\]
where $A^1:=A$ and $A^{0}:=A^c$.
\end{lem}
\begin{proof}
Define a map $\phi: X \to \{0,1\}^\N$ by
\[\phi(x)_n: = 1_A(T^n x) \text{ for any } n\in\N.\]
The map $\phi$ is measurable and so we can define $\nu$ to be the pushforward of $\mu$ under $\phi$, i.e. $\nu(B):= \mu(\phi^{-1}(B))$ for any $B\subset \{0,1\}^\N$ measurable.  For any cylinder $[\omega_{n_1} = i_1,\dots, \omega_{n_r} = i_r]$ we have that
\begin{align*}
\phi^{-1}([\omega_{n_1} = i_1,\dots, \omega_{n_r} = i_r]) &= \{ x\in X \, | \, T^{n_j}x \in A^{i_j} \text{ for each } j=1,\dots ,r\}\\
&=T^{-n_1} A^{i_1} \cap \dots \cap T^{-n_r} A^{i_r}
\end{align*}
and so the result follows.
\end{proof}
\begin{defn}
Let $(F_N)$ be a F{\o}lner sequence.
A point $x\in \{0,1\}^\N$ is called $(F_N)$-generic for a measure $\nu \in \mathcal{M}_\sigma$ if
\begin{equation*}
\delta_{F_N}(x):= \frac{1}{|F_N|}\sum_{i\in F_N} \delta_{\sigma^i(x)} \to \nu \quad \text{as } N \to \infty,
\end{equation*}
where $\delta_{\sigma^i(x)}$ is the Dirac mass at $\sigma^i(x)$.
\end{defn}

Theorem \ref{Theorem: Inverse CP} now follows from Lemma \ref{Lemma: Universality of shift space} combined with the following classical result of Sigmund.
\begin{pro}[\cite{Sigmund}{[Theorem 4]}] \label{Prop: Sigmund Theorem}
Every $\nu \in \mathcal{M}_\sigma$ admits a $(\{0,\dots,N-1\})$-generic point.
\end{pro}
In fact Sigmund proved that any measure on a compact metric space invariant under a continuous map satisfying a property known as specification (of which the shift space $(\{0,1\}^\N,\sigma)$ is an example) admits a $(\{0,\dots,N-1\})$-generic point. For the sake of completeness we include a proof of Proposition \ref{Prop: Sigmund Theorem} in Section \ref{Sec: Proof of Sigmund}.
\begin{proof}[Proof of Theorem \ref{Theorem: Inverse CP}]
Let $(X,\mu,T)$ be a measure preserving system and let $A\subset X$. By Lemma \ref{Lemma: Universality of shift space} there exists some $\nu \in \mathcal{M}_\sigma$ with
\[ \mu\left(\bigcap_{i=1}^k T^{-n_i}A\right) = \nu([\omega_{n_1}=1,\dots,\omega_{n_k}=1])\] 
for all $k,n_1,\dots,n_k\in \N$. Proposition \ref{Prop: Sigmund Theorem} ensures there exists a point $x\in \{0,1\}^\N$ for which
\[\frac{1}{N}\sum_{i=0}^{N-1} \delta_{\sigma^i(x)}([\omega_{n_1}=1,\dots,\omega_{n_k}=1]) \to \nu([\omega_{n_1}=1,\dots,\omega_{n_k}=1])\]
as $N\to \infty$ for all $k,n_1,\dots,n_k\in \N$. Identifying $x$ with the indicator function $1_E$ for some set $E\subset \N$ allows us to re-write
\[\frac{1}{N}\sum_{i=0}^{N-1} \delta_{\sigma^i(x)}([\omega_{n_1}=1,\dots,\omega_{n_k}=1]) = \frac{\left|\left( \bigcap_{i=1}^k (E-n_i) \right)\cap \{0,\dots,N-1\}\right|}{N}\]
for any $k,n_1,\dots,n_k\in \N$ and so the result follows.
\end{proof}

\section{\textbf{Proof of uniformity of recurrence and nice recurrence}}
\begin{proof}[Proof of Theorem \ref{Theorem: Uniformity of recurrence}]
If the result fails then there exists some $\delta>0$, a sequence of positive numbers $(\eta_k)$ decreasing to $0$ and a sequence of positive integers $(N_k)$ increasing to infinity such that the following holds. For every $k\in \N$ there exists a system $(X_k,\mu_k,T_k)$ and a set $A_k\subset X_k$ with $\mu_k(A_k)\geq \delta$ such that
\begin{equation*}
\mu_k(A_k \cap T_k^{-r} A_k) < \eta_k \quad \text{for all }r \in R\cap\{1,\dots,N_k\}.
\end{equation*}
Applying Lemma \ref{Lemma: Universality of shift space} to each $(X_k \supset A_k,\mu_k,T_k)$ yields a sequence of $\sigma$-invariant measures $(\nu_k)$ such that the set $B:=[\omega_0=1]$ has $\nu_k(B)=\mu_k(A_k)\geq \delta$ and 
\begin{equation*}
\nu_k(B \cap \sigma^{-r} B) < \eta_k \quad \text{for all }r \in R\cap\{1,\dots,N_k\}
\end{equation*}
and all $k\in \N$. By compactness of $\mathcal{M}$ there exists some weak-* limit $\nu$ of the $\nu_k$'s with $\nu(B)\geq \delta$, since $B$ is itself a cylinder, and such that
\[ \nu(B\cap T^{-r}B)=0 \quad \text{for every non-zero }r\in R.\]
Since the cylinders generate the Borel $\sigma$-algebra on $\{0,1\}^\N$, convergence of cylinders ensures that the weak-* limit of $\sigma$-invariant measures is $\sigma$-invariant, and so the previous equation contradicts that $R$ is a set of recurrence. 
\end{proof}
\begin{proof}[Proof of Theorem \ref{Theorem: Uniformity of nice recurrence}]
If the result fails then there exists some $\delta,\varepsilon>0$ with $\delta^2> \varepsilon$ and for every $k\in \N$ a system $(X_k,\mu_k,T_k)$ with a set $A_k\subset X_k$ satisfying $\mu_k(A_k)\geq \delta$ for which
\[ \mu_k(A_k \cap T_k^{-r} A_k) \leq \delta^2 -\varepsilon \quad \text{for all } r \in R\cap \{1,\dots,k\}.\]
Proceeding identically to the proof of Theorem \ref{Theorem: Uniformity of recurrence} then we obtain a measure $\nu \in \mathcal{M}_\sigma$ and a set $B\subset \{0,1\}^\N$ with $\nu(B)\geq \delta$ such that
\[ \nu(B\cap \sigma^{-r}B)\leq \delta^2 - \varepsilon \quad \text{for all non-zero }r\in R\]
contradicting that $R$ is a set of nice recurrence.
\end{proof}

\section{\textbf{Existence of }$(\{0,\dots,N-1\})$\textbf{-generic points for shift invariant measures}}\label{Sec: Proof of Sigmund}
Let $\mathcal{M}_\sigma^e \subset \mathcal{M}_\sigma$ be the set of all ergodic $\sigma$-invariant measures. We define the domain of a cylinder $[\omega_{n_1} = i_1,\dots,\omega_{n_k}=i_k]$ to be the set
\[\mathrm{dom}([\omega_{n_1} = i_1,\dots,\omega_{n_k}=i_k]):=\{n_1,\dots,n_k\},\]
and denote by $\mathcal{C}_N$ to be the collection of all cylinders $C$ with $\mathrm{dom}(C)\subset \{0,1,\dots,N-1\}$.

\begin{lem}\label{Lemma: Points close to periodic points}
Let $\nu \in \mathcal{M}_\sigma$. For every $k\in \N$ and every $\varepsilon>0$ there exists a point $x\in \{0,1\}^\N$ and some positive integer $R_0=R(k,\varepsilon,\nu)$ such that for all $R\geq R_0$ any $y \in \{0,1\}^\N$ starting with the word $x_0 x_1 \dots x_{R-1}$ we have that
\begin{equation}
\left| \frac{1}{R}\sum_{i=0}^{R-1} \delta_{\sigma^i(y)}(C) - \nu(C) \right| < \varepsilon \quad \text{for all } C \in \mathcal{C}_k.
\end{equation}
\end{lem}
\begin{proof}
Let $\nu \in \mathcal{M}_\sigma$, $k\in \N$ and $\varepsilon>0$ be given. The Krein-Milman theorem implies that
\[ \overline{\mathrm{conv}(\mathcal{M}_\sigma^e)} = \mathcal{M}_\sigma\]
where $\mathrm{conv}(\mathcal{M}_\sigma^e)$ is the set of all finite convex combinations of ergodic $\sigma$-invariant measures. By finiteness of $\mathcal{C}_k$ it follows that we can find some convex weights $\alpha_1,\dots,\alpha_r \in [0,1]$ and measures $\nu_1,\dots,\nu_r \in \mathcal{M}_\sigma^e$ such that
\begin{equation}\label{eq: 1}
\left| \nu(C) - \sum_{i=1}^r \alpha_i \nu_i(C)\right| < \frac{\varepsilon}{5} \quad \text{for all }C \in \mathcal{C}_k.
\end{equation}
For each $i=1,\dots,r$ use the pointwise ergodic theorem to find a $\nu_i$-generic point $x^i \in \{0,1\}^\N$, i.e. a point for which the measure
\[\delta_n(x^i):= \frac{1}{n}\sum_{j=0}^{n-1}\delta_{\sigma^j (x^i)}\quad \text{has that}  \quad \delta_n(x^i) \to \nu_i \text{ as }n\to \infty.\]
Finiteness of $\mathcal{C}_k$ then ensures that for each $i=1,\dots,r$ there exists some $\tilde{n}_i \in \N$ such that
\begin{equation}\label{eq: 2}
\left|\nu_i(C) - \delta_n(x^i)(C) \right| < \frac{\varepsilon}{5}\quad \text{for all }C \in \mathcal{C}_k \text{ and } n \geq \tilde{n}_i. 
\end{equation}
Now pick some $n_1,\dots,n_r$ satisfying
\begin{equation}\label{eq: 3}
n_i\geq \tilde{n}_i \quad \text{and}\quad  \left| \frac{n_i}{\sum_{j=1}^r n_j}-\alpha_i \right| < \frac{\varepsilon}{5r} \quad \text{for each }i=1, \dots ,r, 
\end{equation}
and such that
\begin{equation}\label{eq: 4}
\frac{rk}{\sum_{j=1}^r n_j} < \frac{\varepsilon}{5}.
\end{equation}
For a point $z\in \{0,1\}^\N$, denote by $z|_{n} := z_0 z_1 \dots z_{n-1}$ the word formed by the first $n$ symbols of $z$ and let $x \in \{0,1\}^\N$ be the point obtained by repeating the finite string
\[x^1|_{n_1}\bullet x^2|_{n_2} \bullet \dots \bullet x^r|_{n_r}\]
where $\bullet$ denotes concatenation. Thus $x$ is a periodic point with period $p: = \sum_{j=1}^r n_j$.
Notice that for any $C\in \mathcal{C}_k$,
\begin{align*}
\delta_{p}(x)(C) &= \frac{1}{p}\sum_{i=1}^r \left[ n_i  \delta_{n_i}(x^i)(C)+e_i  \right] \\
&= \sum_{i=1}^r \frac{n_i}{p}\delta_{n_i}(x^i)(C) + \frac{1}{p}\sum_{i=1}^r e_i.
\end{align*}
where each $e_i$ is an error term resulting from the $i^\text{th}$ concatenation boundary. Since each $C\in \mathcal{C}_k$ has $\mathrm{dom}(C)\subset\{0,1,\dots,k-1\}$ then each $|e_i| \leq k$, and so by equation \eqref{eq: 4} we have that
\begin{equation}\label{eq: 5}
\left| \delta_p(x)(C) - \sum_{i=1}^r  \frac{n_i}{p}  \delta_{n_i}(x^i)(C) \right| <\frac{\varepsilon}{5}\quad \text{for all } C \in \mathcal{C}_k. 
\end{equation}
Since each $n_i \geq \tilde{n}_i$ then we may combine our estimates from equations \eqref{eq: 1}, \eqref{eq: 2}, \eqref{eq: 3} and \eqref{eq: 5} to conclude that
\begin{equation} \label{eq: 6}
\left| \delta_p(x)(C) - \nu(C)\right| < \frac{4\varepsilon}{5} \quad \text{for all } C \in \mathcal{C}_k.
\end{equation}
Since $x$ is periodic with period $p$ then $\delta_p(x) = \delta_{qp}(x)$ for any positive integer $q$. Picking $R_0 = pq$ for $q$ sufficiently large in terms of $p,\varepsilon$ and $k$ will then ensure that for any $R\geq R_0$ and any point $y\in \{0,1\}^\N$ starting with the word $x_0 x_1 \dots x_{R-1}$ will have that 
\[\left| \delta_{R}(y)(C) - \delta_{p}(x)(C) \right| < \frac{\varepsilon}{5} \quad \text{for all } C \in \mathcal{C}_k,\]
which combined with equation \eqref{eq: 6} completes the proof.
\end{proof}
Notice that the proof of Lemma \ref{Lemma: Points close to periodic points} in particular implies periodic measures, i.e. those giving equal mass to each point in finite $\sigma$-orbit of a periodic point, are dense in $\mathcal{M}_\sigma$. The proof of Proposition \ref{Prop: Sigmund Theorem} proceeds by concatenating larger and larger pieces of these periodic points in the correct proportion.
\begin{proof}[Proof of Proposition \ref{Prop: Sigmund Theorem}]
Let $\nu \in \mathcal{M}_\sigma$ and use Lemma \ref{Lemma: Points close to periodic points} to obtain for each $k\in \N$ a point $x^k \in \{0,1\}^\N$ and a positive integer $R_k$ such that for every $R\geq R_k$ and every $y\in \{0,1\}^\N$ with $y|_{R} = x^k|_R$ we have that
\begin{equation}\label{eq: points from the lemma}
\left| \delta_R(y)(C) - \nu(C)\right| < \frac{1}{k} \quad \text{for all }C \in \mathcal{C}_k.
\end{equation}
Now let $L_1 := R_1$ and for $j\geq 2$ inductively pick some integer $L_j$ such that
\begin{align} 
&L_j\geq R_j, \label{eq: (i)} \\
&\frac{R_{j+1}}{L_j} < \frac{1}{j}, \text{ and} \label{eq: (ii)}\\
&\frac{L_j}{L_1+L_2+\dots+L_j}> 1 - \frac{1}{j}.\label{eq: (iii)}
\end{align}
We claim that the point
\[ y:= x^1|_{L_1} \bullet x^2|_{L_2} \bullet x^3|_{L_3} \dots \]
has that $\delta_{N}(y)\to \nu$ as $N \to \infty$. Indeed for any $N$ let $k=k(N)$ be such that
\[ L_1+\dots+L_k \leq N < L_1+\dots+L_{k+1}\]
and write $N= L_1+\dots+L_k+L$. Notice that we can then split $\delta_N(y)$ into three terms which we will refer to as first, middle and last for the remainder of the proof. The three terms are
\[ \delta_N(y) = \frac{1}{N} \sum_{i=1}^{k-1}L_i \delta_{L_i}(y^i) + \frac{L_k}{N} \delta_{L_k}(y^k) + \frac{L}{N} \delta_L(y^{k+1})\]
where $y^i\in \{0,1\}^\N$ has $y^i|_{L_i} = x^i|_{L_i}$ for each $i=1,\dots,k$ and $y^{k+1}|L = x^{k+1}|_{L}$. Equation \eqref{eq: (iii)} always ensures that the first term is $\mathcal{O}(1/k)$. If $L < R_{k+1}$ then equations \eqref{eq: (ii)} and \eqref{eq: (iii)} ensure both that the last term is $\mathcal{O}(1/k)$ and that $L_k/N = 1+\mathcal{O}(1/k)$. This observation combined with equations \eqref{eq: (i)} and \eqref{eq: points from the lemma} ensure that for every cylinder $C \in \mathcal{C}_k$ the middle term has that
\begin{align*}
\frac{L_k}{N} \delta_{L_k}(y^k)(C) &= \left( 1+ \mathcal{O}(1/k)\right)\left( \nu(C) + \mathcal{O}(1/k) \right)\\
&=\nu(C) + \mathcal{O}(1/k).
\end{align*}
If instead $L\geq R_{k+1}$ then the last term may not be negligible but instead equation \eqref{eq: points from the lemma} applies to both the middle and the final term to see that for every cylinder $C\in \mathcal{C}_k$ we have that
\begin{align*}
\frac{L_k}{N} \delta_{L_k}(y^k)(C) + \frac{L}{N} \delta_L(y^{k+1})(C) &= 
\frac{L_k+L}{N} \left(\nu(C) + \mathcal{O}(1/k)\right) \\
&= \left(1+\mathcal{O}(1/k)\right)\nu(C) + \mathcal{O}(1/k)
\end{align*}
where in the last line we use equation \eqref{eq: (iii)}. In any case then we have that
\[ \delta_{N}(y)(C) = \nu(C)+\mathcal{O}(1/k)\quad \text{for all } C\in\mathcal{C}_k.\]
Since $k=k(N)\to \infty$ as $N\to \infty$ then for any fixed cylinder $C$ this implies that
\[\delta_N(y)(C) \to \nu(C) \quad \text{as}\quad N\to \infty\]
as required.
\end{proof}

\section{\textbf{Further questions}}
As stated our inverse correspondence principle only applies to the F{\o}lner sequence $(F_N)=(\{0,1,\dots,N-1\})$, and so Questions \ref{Question: Characterising nice recurrence} and \ref{Question: Inverse CP} both remain open in their full generality. Although the construction in the proof of Proposition \ref{Prop: Sigmund Theorem} is flexible enough to also cover the cases when $(F_N)$ is an increasing sequence of nested intervals or when $(F_N)$ is a sequence of disjoint intervals, the general case when $(F_N)$ is an arbitrary F{\o}lner sequence remains unclear. Indeed we do not even know how to handle the case when
\[ F_N = \{1,\dots,n_{N}\} + t_N\]
for some increasing sequence $(n_N)$ and some shifts $t_N \in \N$.

In a slightly different direction, if in Definition \ref{Def: Recurrence} we replace the words "for any measure preserving system" with "for any \textit{ergodic} measure preserving system", then we obtain the a-priori weaker definition of a set of \textit{ergodic recurrence}. In the exact same way one can modify Definition \ref{Def: Nice Recurrence} to obtain the definition of a set of \textit{ergodic nice recurrence}. It is known that being a set of ergodic recurrence is in fact equivalent to being a set of recurrence, indeed one proof follows from Theorem \ref{Theorem: Uniformity of recurrence} and an appeal to the ergodic decomposition. Whether or not the same equivalence holds for nice recurrence is another question of Moreira \cite{Moreira}.
\begin{qu}\label{Question: ergodic nice recurrence}
Does there exist a set of ergodic nice recurrence which is not a set of nice recurrence?
\end{qu}
Contrastingly to all other similarities between sets of recurrence and sets of nice recurrence, both proved and conjectured, we expect that the answer to Question \ref{Question: ergodic nice recurrence} is yes because in the setting of multiple recurrence the analogous version of Question \ref{Question: ergodic nice recurrence} is known to have an affirmative answer. Indeed the authors of \cite{BHKR} showed that if $(X,\mu,T)$ is an ergodic measure preserving system and $A\subset X$ has positive measure, then for any $\varepsilon>0$ the set
\[ R_{3}(A,\varepsilon): = \{n\in \N \, : \, \mu(A\cap T^{-n}A \cap T^{-2n})>\mu(A)^3-\varepsilon\} \]
is syndetic and in particular non-empty, however they also constructed an example of a non-ergodic system $(X,\mu,T)$ with a positive measure set $A\subset X$ for which $R_3(A,\varepsilon)$ is empty. In the setting of single recurrence however, Question \ref{Question: ergodic nice recurrence} remains open.


\begin{thebibliography}{99}
\bibitem{Ack} Ackelsberg, E. M. (2021). \emph{Rigidity, weak mixing, and recurrence in abelian groups.} arXiv preprint arXiv:2110.01802.
\bibitem{B} Bergelson, V. (1986). \emph{A density statement generalizing Schur's theorem.} Journal of Combinatorial Theory, Series A, 43(2), 338-343.
\bibitem{BH} Bergelson, V., \& H$\mathring{\text{a}}$land, I. J. (1996). \emph{Sets of recurrence and generalized polynomials.} Convergence in ergodic theory and probability, 5, 91-110.
\bibitem{BHKR} Bergelson, V., Host, B., Kra, B., \& Ruzsa, I. (2005). \emph{Multiple recurrence and nilsequences.} Inventiones mathematicae, 160(2), 261-303.
\bibitem{BHMP} Bergelson, V., Host, B., McCutcheon, R., \& Parreau, F. (2000). \emph{Aspects of uniformity in recurrence.} In Colloquium Mathematicum (Vol. 84, No. 2, pp. 549-576). Polska Akademia Nauk. Instytut Matematyczny PAN.
\bibitem{ForThesis} Forrest, A. H. (1990). \emph{Recurrence in dynamical systems: A combinatorial approach}, ProQuest LLC. Ann Arbor, MI.

\bibitem{FraKra} Frantzikinakis, N., \& Kra, B. (2006). \emph{Ergodic averages for independent polynomials and applications.} Journal of the London Mathematical Society, 74(1), 131-142.
\bibitem{FLW} Frantzikinakis, N., Lesigne, E., \& Wierdl, M. (2009). \emph{Powers of sequences and recurrence.} Proceedings of the London Mathematical Society, 98(2), 504-530.
\bibitem{Fur2} Furstenberg, H. (1981). \emph{Recurrence in ergodic theory and combinatorial number theory.} M. B. Porter Lectures. Princeton University Press, Princeton, NJ.
\bibitem{KM} Kamae, T., \& Mendès France, M. (1978). Van der Corput’s difference theorem. Israel Journal of Mathematics, 31, 335-342.
\bibitem{Moreira} Moreira, J. (2013) \emph{Sets of nice recurrence}, I Can’t Believe It’s Not Random! Available at: \url{https://joelmoreira.wordpress.com/2013/03/04/323/} (Accessed: 27 July 2024). 

\bibitem{Sa}  Sárközy, A. (1978). \emph{On difference sets of sequences of integers. III.} Acta Mathematica Academiae Scientiarum Hungarica, 31, 355-386.
\bibitem{Sigmund} Sigmund, K. (1974). \emph{On dynamical systems with the specification property.} Transactions of the American Mathematical Society, 190, 285-299.
\bibitem{Tao} Tao, T. (2009). \emph{Poincaré's Legacies, Part II: Pages from Year Two of a Mathematical Blog.} American Mathematical Soc.






\end{thebibliography}
\end{document}